\documentclass{article}
\usepackage{graphicx}
\usepackage{epstopdf}
\usepackage{amsmath}
\usepackage{amssymb}
\usepackage{bm}
\usepackage{bbm}
\usepackage{enumerate}
\usepackage[T1]{fontenc}
\usepackage[latin1]{inputenc}
\usepackage{hyperref}
\usepackage{color}
\usepackage{curves}
\usepackage{tikz}
\usepackage[hang]{caption}
\usepackage{ntheorem}
\usepackage{array,float,caption,nicefrac}
\usepackage{mathrsfs}

\title{The averaging process on infinite graphs}
\author{Nina Gantert
	\\\normalsize Technical University of Munich\\
	Timo Vilkas
		\\\normalsize Lunds Universitet
}

\theoremstyle{break}
\newtheorem{theorem}{Theorem}[section]
\newtheorem{corollary}{Corollary}[section]
\newtheorem{lemma}{Lemma}[section]
\newtheorem{proposition}{Proposition}[section]
\theorembodyfont{\upshape}
\newtheorem{definition}{Definition}

\makeatletter
\let\c@proposition\c@theorem
\let\c@lemma\c@theorem
\let\c@corollary\c@theorem
\makeatother

\newenvironment{proof}{\noindent{\sc Proof:}}{\vspace{-1em}~\hfill $\square$\vspace{2em}}
\newenvironment{nproof}[1]{\noindent{\sc Proof #1:}}{\vspace{-1em}~\hfill $\square$\vspace{2em}}

\newcommand\IN{\mathbb{N}}

\newcommand\IZ{\mathbb{Z}}

\newcommand\IE{\mathbb{E}\,}
\newcommand\Prob{\mathbb{P}}
\renewcommand\epsilon{\varepsilon}
\renewcommand\phi{\varphi}

\definecolor{darkblue}{rgb}{0,0,.5}
\hypersetup{colorlinks=true, breaklinks=true, linkcolor=blue, 
citecolor=darkblue, menucolor=blue, urlcolor=blue}

\begin{document}
\newpage
\maketitle
\begin{abstract}

\end{abstract}
We consider the averaging process on an infinite connected graph with bounded degree and independent, identically distributed starting values or initial opinions. Assuming that the law of the initial opinion of a vertex has a finite second moment, we show that the opinions of all vertices converge in $L^2$ to the first moment of the law of the initial opinions. A key tool in the proof is the {\em Sharing a drink} procedure introduced by Olle H\"aggstr\"om.\\

\noindent
\textbf{Keywords:} 
Averaging process, law of large numbers, sharing a drink procedure

\section{Introduction}
The averaging process is given as follows: all nodes of a graph carry a (real) value, which we think of as an opinion. The edges are updated according to Poisson clocks, and when an edge is updated, the opinions of the two incident vertices average (or make a step towards each other), see Definition \ref{basicdef} for a precise description. On a finite connected graph, it is easy to see that the opinions of all vertices converge to the sample mean of the starting values (note that the sum of all opinions is preserved by the dynamics).
A lot of recent research is dedicated to investigate the speed of convergence, see for instance \cite{cutoff, repeated1, repeated2}.
It is natural to conjecture that for an infinite graph with bounded degree and i.i.d. starting values with finite mean $m_1$ the opinions of all vertices converge almost surely to $m_1$.
An analogous statement was proved recently for the DeGroot dynamics in \cite{deGroot}.
However, whether or not the conjecture regarding the averaging process holds is still open. In this note, we make a step towards proving it, by showing that for an infinite graph with bounded degree and i.i.d. starting values with a finite second moment, the opinions of all vertices converge in $L^2$ to $m_1$.
Our main result is Theorem \ref{avgthm}. A key tool is the non-random pairwise averaging process H\"aggstr\"om \cite{ShareDrink}
proposed to call {\em Sharing a drink} (SAD) on a graph $G=(V,E)$, see Section \ref{Preliminaries}.
Prompted by a question of ours, Gollin et al.\ \cite{SAD} extended an upper bound for the maximal amount of liquid in SAD (previously known for trees) to general graphs, see Theorem \ref{globalbound}. Our proof uses their result.

\begin{definition}\label{basicdef}
 Consider a simple (connected) graph $G=(V,E)$ with bounded degree, together with an initial profile $\{\eta_0(v)\}_{v\in V}$, where $\eta_t(v)$ is
 the value associated with vertex $v\in V$ at time $t$. The {\em averaging process} consists of updates of the profile
 along single edges of the following form: If $\eta$ is updated along the edge $e=\langle u,w \rangle$ to the extent $\mu\in(0,\frac12]$, the resulting profile $\eta'$ is given by
 
 \begin{equation}\label{average}\begin{aligned}\eta'(u)&=(1-\mu)\,\eta(u)+\mu\,\eta(w),\\
 		\eta'(w)&=\mu\,\eta(u)+(1-\mu)\,\eta(w),\\
 		\eta'(v)&=\eta(v)\quad\text{for all }v\notin\{u,w\}.\end{aligned}
 \end{equation}
In discrete time, the process can be based on a deterministic sequence of edges for the updates or they can be chosen uniformly at random (for finite $G$). In continuous time, one can use independent
homogeneous Poisson point processes having the same intensity $\lambda$ associated with the edges, where an event on edge $e=\langle u,w \rangle$ results in an update of the incident values according to \eqref{average}. The parameter $\mu$ in the convex combination can be chosen to be $\frac12$ (complete averaging), more generally constant or even different for each update.
\end{definition}

In fact, for finite $G$, using independent Poisson-processes to determine edges along which updates take place is a time-continuous way of picking the next edge uniformly at random.
Given countable $V$, also the edge set $E$ is countable and there will almost surely be neither simultaneous Poisson events nor a limit point in time for the Poisson events on edges incident to one fixed vertex.
By that, the well-definedness of the process by \eqref{average} is guaranteed for finite $G$. The
extension to infinite graphs with bounded degree is not immediately obvious but follows from a standard argument,
see e.g.\ Thm.\ 3.9 in \cite{Liggett}. Note that the assumption of bounded degree enters already at this place.

We will work with both the time-continuous process for general initial profiles and in the special case when
$\eta_0$ has only finitely many non-zero values the corresponding embedded time-discrete process (for which the updates involving non-zero incident values are enumerated in chronological order).

Our main results is the following.
Consider the averaging process as described in Definition \ref{basicdef} on an infinite
(connected) graph $G=(V,E)$ with bounded degree in continuous time. We will assume that the random variables
 $\{\eta_0(v), v\in V\}$ are i.i.d. and we will denote
the expected value of the initial profile at some (or any) vertex $v$ by
\begin{equation}\label{defmom}
m_1: = \IE(\eta_0(v)) 
\end{equation}
and the second moment by
\begin{equation}\label{defsecmom}
m_2: = \IE(\eta_0(v)^2) \, .
\end{equation}

\begin{theorem}\label{avgthm}
	Consider the averaging process in continuous time on an infinite graph $G=(V,E)$ with initial profile $\{\eta_0(v)\}_{v\in V}$. Assume that the random variables\\
 $\{\eta_0(v), v\in V\}$ are i.i.d.\ and the marginal distribution $\eta_0(v)$ has a finite second moment. Then it holds
	\[\lim_{t\to\infty}\eta_t(v)=m_1 \quad\text{ in }L^2 \quad\text{ for all }v\in V.\]
\end{theorem}

\section{Preliminaries}\label{Preliminaries}

A key concept is the non-random pairwise averaging process H\"aggstr\"om \cite{ShareDrink}
proposed to call {\em Sharing a drink} (SAD) on a graph $G=(V,E)$.

Glasses are put, one at each vertex: the one at site $r\in V$ is full, all others are empty. 
As time proceeds neighbors interact and share. To be more precise, the procedure starts with the initial profile $\xi_0=\delta_r$, i.e.\ $\xi_0(r)=1$ and $\xi_0(v)=0$ for all $v\neq r$. In each step, an edge is selected along which the two incident vertices share their water in the same way as described in \eqref{average}: Given profile
$\{\xi_n(v)\}_{v\in\IZ}$ after step $n$, an update along the edge $\langle u,w\rangle$ leads to
\begin{equation}\label{update}\begin{aligned}\xi_{n+1}(u)&=(1-\mu)\,\xi_{n}(u)+\mu\,\xi_{n}(w),\\
		\xi_{n+1}(w)&=\mu\,\xi_{n}(u)+(1-\mu)\,\xi_{n}(w),\\
		\xi_{n+1}(v)&=\xi_{n}(v),\ \text{for all }v\notin\{u,w\}.\end{aligned}
\end{equation}
For arbitrary $n\in\IN_0$, the result $\{\xi_n(v)\}_{v\in V}$ of $n$ updates involving non-empty glasses applied to 
$\xi_0=\delta_r$, will be called an {\em SAD-profile}. Note that these profiles have only finitely many non-zero values, which are all positive and sum to 1. With $d(u,v)$ denoting the graph distance between vertices $u$ and $v$ on $G$,
the following upper bound was recently extended from trees to general graphs by Gollin et al.\ \cite{SAD} using a simple but clever combinatorial argument:
\begin{theorem}\label{globalbound}
	Consider the SAD-process on an arbitrary graph $G=(V,E)$ started in vertex $r$, i.e.\ with
	$\xi_0(v)=\delta_r(v),\ v\in V$. Then the amount at any fixed vertex $v\in V$ is bounded from above by
	$\frac{1}{d(r,v)+1}$.
\end{theorem}

\begin{proof}
We first consider the case $\mu = \frac12$ and follow the proof of 	Gollin et al.\ \cite{SAD}. Let the function $f: 2^V\to[0,1]$ be defined by
	\[f(S)=\prod_{v\in S}\limits\frac{d(r,v)}{d(r,v)+1},\quad\text{for }S\subseteq V.\]
	It is worth noting two simple facts: On the one hand $f$ is multiplicative, in the sense that $f(A\cup B)=f(A)\cdot f(B)$ for disjoint vertex sets $A,B\subseteq V$, and on the other it is enough to show that for arbitrary SAD-profile $\{\xi(v)\}_{v\in V}$ it holds
	\begin{equation}\label{genineq}
		\sum_{v\in S}\xi(v)\leq 1-f(S),
	\end{equation}
	as the claim then immediately follows by choosing $S=\{v\}$.
	
	Proceeding by induction on the number of rounds, for $\xi_0=\delta_r$ inequality \eqref{genineq} trivially holds since $f(S)=0$ if $r\in S$. If $\xi_{n+1}$ arises from $\xi_{n}$ by sharing the drink along the edge $e=\langle u,w\rangle$, i.e.\
	\[\xi_{n+1}(v)=\begin{cases}
		\frac12\big(\xi_n(u)+\xi_n(w)\big),& \text{for } v\in\{u,w\}\\
		\xi_n(v),& \text{otherwise}	
	\end{cases}\]
	the sum in \eqref{genineq} is only changed for subsets $S$ containing exactly one vertex incident to $e$. Without loss of generality, consider $u\in S,\ w\notin S$. Then the sum can be rewritten as
	\begin{align*}\sum_{v\in S}\xi_{n+1}(v)&=\frac12\bigg(\sum_{v\in S\cup\{w\}}\xi_n(v)+\sum_{v\in S\setminus\{u\}}\xi_n(v)\bigg)\\
		&\leq \frac12 \Big(1-f\big(S\cup\{w\}\big)+1-f\big(S\setminus\{u\}\big)\Big)\\&=1-\frac{f\big(S\setminus\{u\}\big)}{2}\cdot \Big(f\big(\{u,w\}\big)+1\Big),
	\end{align*}
	where both the induction hypothesis and the multiplicativity of $f$ were used. To conclude, note that $d(r,u)\leq d(r,w)+1$ as $u$ and $w$ are neighbors and hence
	\[f\big(\{u,w\}\big)+1=\tfrac{d(r,u)}{d(r,u)+1}\cdot\big(1-\tfrac{1}{d(r,w)+1}\big)+1\geq\tfrac{d(r,u)-1}{d(r,u)+1}+1=2\,f\big(\{u\}\big).\]
	As a consequence, \eqref{genineq} holds for $\xi_{n+1}$ as well which concludes the proof for $\mu=\frac12$. In order to extend the bound to general $\mu\in(0,\frac12]$ the two lemmas below are needed.
\end{proof}

Of course, we can consider other initial profiles than $\eta_0=\delta_r$. Given a finite graph $G$ and a time horizon $t>0$, let $\phi=([e_n,\mu_n])_{1\leq n \leq n(t)}$ be the sequence of edges in $G$, together with the corresponding weights $\mu_n\in(0,\frac12]$ encoding the update steps (in chronological order) until time $t$ (in discrete time $n(t)=t$). Then the profile resulting from the SAD-process started in a given vertex $r\in V$ with reversed time order describes the contributions of all vertices to the current value at $r$.

\begin{lemma}[Duality]\label{dual}
	Consider an initial profile $\{\eta_0(v)\}_{v\in V}$ on a finite graph $G=(V,E)$, together with a sequence 
	$\phi$ of edges and weights encoding the update steps, and fix a vertex $r\in V$. For $n=n(t)\in\IN_0$ we define the SAD-process dual to the update sequence $\phi$ as follows: Starting with $\xi_0=\delta_r$ the profile is updated according to \eqref{update} but with respect to $\overleftarrow{\phi}_t=([e_n,\mu_n],\dots,[e_1,\mu_1])$. Then it holds
	\begin{equation}\label{convcomb}
		\eta_t(r)=\sum_{v\in V}\xi_t(v)\,\eta_0(v).
	\end{equation}
\end{lemma}

\begin{proof}
	Without loss of generality, we can consider the embedded discrete time process and use $n(t)=n$ to ease notation a little
	(as the number $n(t)$ of updates until time $t$ is a.s.\ finite in the continuous-time version on a finite graph). We prove the statement by induction on $n$. For $n=0$, the statement is trivial.
	For the induction step, fix $n\in\IN$ and assume the first update to be on $e_1=\langle u,w\rangle$. According to 
	(\ref{update}) we get
	$$\eta_1(v)=\begin{cases}\eta_0(v)& \text{if }v\notin\{u,w\}\\
		(1-\mu_1)\,\eta_0(u)+\mu_1\,\eta_0(w)&\text{if }v=u\\
		(1-\mu_1)\,\eta_0(w)+\mu_1\,\eta_0(u)&\text{if }v=w.
	\end{cases}$$
	Let us treat $\{\eta_1(v)\}_{v\in V}$ as an initial profile $\{\eta'_0(v)\}_{v\in V}$. By induction
	hypothesis we know
	\begin{align*}
		\eta'_{n-1}(r)&=\sum_{v\in V}\xi'_{n-1}(v)\,\eta'_0(v)\\
		&=\sum_{v\in V\setminus\{u,w\}}\xi'_{n-1}(v)\,\eta_0(v)
		+\Big((1-\mu_1)\,\xi'_{n-1}(u)+\mu_1\,\xi'_{n-1}(w)\Big)\,\eta_0(u)\\
		&\quad+\Big((1-\mu_1)\,\xi'_{n-1}(w)+\mu_1\,\xi'_{n-1}(u)\Big)\,\eta_0(w),
	\end{align*}
	where $\eta'_{n-1}(v)=\eta_n(v)$ and $\{\xi'_k(v)\}_{v\in V}$, $0\leq k\leq n-1$, are the intermediate
	profiles of the SAD-process corresponding to the sequence $([e_n,\mu_n],\dots,[e_2,\mu_2])$. By definition the original dual SAD-process arises from the shortened one by adding an update along edge $e_1$ with parameter $\mu_1$ at the end. Consequently, we get $\xi_k(v)=\xi'_k(v)$ for all $k\in\{0,\dots,n-1\}$ and $v\in V$ as well as
	$$\xi_n(v)=\begin{cases}\xi_{n-1}(v)=\xi'_{n-1}(v)& \text{if }v\notin\{u,w\}\\
		(1-\mu_1)\,\xi_{n-1}(u)+\mu_1\,\xi_{n-1}(w)&\text{if }v=u\\
		(1-\mu_1)\,\xi_{n-1}(w)+\mu_1\,\xi_{n-1}(u)&\text{if }v=w,
	\end{cases}$$
	which establishes the claim for finite $G$. 
\end{proof}
\begin{corollary}\label{cordual}
	The representation \eqref{convcomb} holds for infinite $G$ as well.
\end{corollary}
\begin{proof}
	For an infinite graph $G=(V,E)$ it is crucial to observe that the subset of nodes $S_t(r):=\{v\in V,\ \xi_t(v)>0\}$ is characterized by the existence of a path $(e_1,\dots,e_m)$ from $r$ to $v$ with sequential updates in reversed order of the edges. Consequently, even on infinite $G$ the set $S_t(r)$ is finite for finitely many updates in discrete time and a.s.\ finite for fixed time $t$ in the time-continuous version of the process: Let the travel time from $r$ to $v$ in first passage percolation (FPP) with i.i.d.\ $\mathrm{Exp}(\lambda)$ waiting times on $G$ be denoted by $T(r,v)$. Then it holds
	$\Prob(v\in S_r(t))=\Prob(T(v,r)\leq t)=\Prob(T(r,v)\leq t)$,
	hence \[\IE(|S_r(t)|)=\IE\sum_{v\in V}\mathbbm{1}_{\{T(r,v)\leq t\}}=\IE(|C_t(r)|)<\infty,\]
	where $C_t(r)$ is the component explored by FPP started in $r$ until time $t$. Given a fixed time horizon $t$ it is therefore enough to consider a (random but) finite subgraph of $G$ to which Lemma \ref{dual} applies.
\end{proof}

\begin{lemma}\label{simplif}
	Consider the SAD-process on $G=(V,E)$ starting from vertex $r\in V$. Then the highest possible level at arbitrary vertex $s\in V$ achievable with at most $n$ updates of the form \eqref{update} is the same as if $\mu$ is restricted to be $\frac12$ in each update. 
\end{lemma}

\begin{proof}
	As before let $\phi=([e_k,\mu_k])_{1\leq k \leq n}$ denote the sequence of (general) updates -- set $\mu_k=0$ for the last elements in the sequence if there are less than $n$ updates -- and $\{\xi_{n}(v)\}_{v\in V}$ be the resulting SAD-profile. In particular, the first update of $\xi_0=\delta_r$ is on $e_1=\langle u,w\rangle$ with weight $\mu_1\in[0,\tfrac12]$ following the notation in \eqref{update}.
	Without loss of generality we can assume $\xi_0(u)\leq\xi_0(w)$, which in turn implies $\xi_0(u)\leq\xi_1(u)\leq\xi_1(w)\leq\xi_0(w)$ as $0\leq\mu_1\leq \frac12$. Consider the profile $\{\xi'_{n-1}(v)\}_{v\in V}$ resulting from the SAD-process started in $s$,  with time-reversed update sequence $([e_n,\mu_n],\dots,[e_2,\mu_2])\in (E\times[0,\tfrac12])^{n-1}$. Then Lemma \ref{dual} resp.\ Corollary \ref{cordual} applied to $\eta_0=\xi_1$ and update sequence $\phi'=([e_2,\mu_2],\dots,[e_n,\mu_n])$ gives
	\[\xi_n(s)=\sum_{v\in V}\xi'_{n-1}(v)\,\xi_1(v).\]
	There are two possible cases: either $\xi'_{n-1}(u)\leq\xi'_{n-1}(w)$ or $\xi'_{n-1}(u)>\xi'_{n-1}(w)$.
In the first case changing $\mu_1$ to $0$, i.e.\ erasing the first update, will not decrease $\xi_n(s)$, the level
finally achieved at $s$. In the second case the same holds for raising $\mu_1$ to
$\tfrac12$. Since we can consider any update in the sequence $\phi$ to be the first one applied to the intermediate
SAD-profile achieved so far, this establishes the claim.
\end{proof}

The following two results are neither striking nor novel, yet we include them as they fit the general theme and contribute to keeping this paper self-contained.

\begin{lemma}[Preservation of the average]\label{averagepreserving}
	Consider the averaging process in continuous time on a graph $G=(V,E)$ with i.i.d.\ initial profile $\{\eta_0(v)\}_{v\in V}$. Then it holds for all $r\in V$ and $t\geq 0$, recalling \eqref{defmom}:
\begin{equation}
	\IE(\eta_t(r))= m_1\, .
\end{equation}
\end{lemma}

\begin{proof}
This is an immediate consequence of \eqref{convcomb}: Fix arbitrary $r\in V$ and let $\mathcal{F}_t$ denote the sigma field generated by all Poisson-processes until time $t\geq0$. Then $\{\xi_t(v),\; v\in V, t\geq 0\}$, the SAD-process dual to the updates (started in $r$), is adapted to the filtration $(\mathcal{F}_t)_{t\geq 0}$ and we can conclude
\begin{align*}\IE\big(\eta_t(r)\big)&=\IE\big(\IE[\eta_t(r)\,|\,\mathcal{F}_t]\big)=\IE\bigg(\sum_{v\in V}\xi_t(v)\cdot\IE[\eta_0(v)\,|\,\mathcal{F}_t]\bigg)\\&=m_1 \cdot\IE\bigg(\underbrace{\sum_{v\in V}\xi_t(v)}_{=1}\bigg),\end{align*}
using the tower property, duality equation \eqref{convcomb} as well as the independence of the Poisson processes and initial configuration. Note that there are no convergence issues as the sum a.s.\ consists of finitely many summands.
\end{proof}

\begin{lemma}\label{finite}
	Consider an initial profile $\{\eta_0(v)\}_{v\in V}$ on a finite (connected) graph $G=(V,E)$ and let $n:=|V|$ denote the number of vertices. Provided that, given a fixed $t\geq0$, there will (a.s.) be updates on all edges after time $t$, the values converge almost surely to the average, i.e.\ for any $v\in V$ it holds
	\[
		\lim_{t\to\infty}\eta_t(v)=\frac{1}{n}\sum_{v\in V}\eta_0(v)\text{\ a.s.}\]
\end{lemma}

\begin{proof}
Consider $X_t:=\max\{\eta_t(v),\;v\in V\}$ and $Y_t:=\min\{\eta_t(v),\;v\in V\}$. Since $(X_t)_{t\geq 0}$ is decreasing, $(Y_t)_{t\geq 0}$ increasing and $X_t\geq Y_t$, trivially both limits $X:=\lim X_t$ and $Y:=\lim Y_t$ exist. As $\sum_{v\in V}\eta_t(v)<\infty$ is constant in time, what needs to be shown is that $X=Y$ almost surely.

Define the total energy at time $t$ by $W_t=\sum_{v\in V}\eta_t(v)^2$ and note that it is decreasing in $t$ as an update on $e=\langle u,v\rangle\in E$ with incident values $(\eta_t(u),\eta_t(v))$ reduces the total energy by $2\mu(1-\mu)(\eta_t(u)-\eta_t(v))^2$.

Assume for contradiction that there exists $\epsilon >0$ such that $X-Y\geq\epsilon$ has positive probability. By monotonicity we can conclude $X_t-Y_t\geq X-Y$ and further that given $X-Y\geq\epsilon$ there exists a subinterval $[a,b]\subseteq(Y_t,X_t)$ of length $b-a\geq\frac{\epsilon}{n}$, such that $\eta_t(v)\notin [a,b]$ for all $v\in V$; the $n$ values $\{\eta_t(v),\, v\in V\}$ simply can't cover $[X_t,Y_t]$ with a finer mesh. This subinterval splits the nodes in two non-empty groups: those with values $\eta_t(v)<a$ and those with $\eta_t(v)>b$. Since $G$ is connected, all edges are updated after time $t$ and values can only enter $[a,b]$ when two neighbors not belonging to the same group average, eventually there will be such an update, causing a loss of energy of at least  $2\mu(1-\mu)(\frac{\epsilon}{n})^2$.
However, since $W_0$ is (a.s.) finite and $\lim W_t\geq 0$, this can happen only a finite number of times, forcing $\{X-Y\geq\epsilon\}$ to be a null-set. This contradicts our assumption and concludes the proof.
\end{proof}

We point out that for the continuous-time version (in particular when updated edges are chosen uniformly at random) the condition that a.s.\ all edges are updated after any given time $t$ is met.

\section{Proof of Theorem \ref{avgthm}}\label{avg}

From \eqref{convcomb}, we know that the water levels of the SAD-process started in $v$ dual to
the update sequence until time $t$ describe the contributions to $\eta_t(v)$, which denotes the value associated with vertex $v$ at time $t$ and is a convex combination of the initial values $\{\eta_0(v)\}_{v\in V}$. As different target vertices $v$ will come into play (and as a consequence dual SAD-processes started in different vertices have to be considered) let us adapt the notation by writing $\xi_t(u,v)$ for the contribution of $\eta_0(u)$ to $\eta_t(v)$, which of course corresponds to $\xi_t(u)$ in the dual SAD-process started at $v$. Then the duality equation \eqref{convcomb} reads
\[\eta_t(v)=\sum_{u\in V}\xi_t(u,v)\,\eta_0(u).\]
As updates lead to convex combinations of the incident values preserving the sum, it is easy to verify that for arbitrary $u,v\in V$ and $t>0$ it holds
\[\sum_{v\in V}\xi_t(u,v)=\sum_{u\in V}\xi_t(u,v)=1\]
and from Theorem \ref{genineq} we know that $\xi_t(u,v)\leq \frac{1}{d(u,v)+1}$.

In the sequel, we will utilize the fact that $\xi_t(u,v)$, the coefficient of $\eta_0(u)$ in the sum representing
$\eta_t(v)$ as a convex combination of initial values, can be seen either as the value at vertex $v$ at time $t$ in a simultaneous averaging process with the same updates but initial profile $\delta_u$ or as the water level at vertex $u$ of the SAD-process started in $v$ that is dual to the update sequence during the time interval $[0,t]$, cf.\ \eqref{convcomb}.

On a transitive graph $G$, the mass-transport principle (cf.\ \cite{MTP}, p.\ 43) immediately gives for fixed $u,v\in V$
\[\IE\big[\xi_t(u,v)\big]=\IE\big[\xi_t(v,u)].\]
Owing to the fact that the evolution of the process is based on independent Poisson-processes, an even stronger statement holds in fact for general $G$:

\begin{proposition}
	Consider the averaging process in continuous time on graph $G=(V,E)$ and for fixed $u,v\in V$, let $\xi_t(u,v)$ be the coefficient of $\eta_0(u)$ in $\eta_t(v)=\sum_{u\in V}\limits\xi_t(u,v)\,\eta_0(u)$. Then it holds
	\[\xi_t(u,v)\stackrel{d}{=}\xi_t(v,u).\]
\end{proposition}

\begin{proof}
	Given the observation above that $\xi_t(u,v)$ emerges both as the water level at $v$ of the SAD-process started in $u$ and the water level at $u$ of the dual SAD-process (i.e.\ with the updates during $[0,t]$ reversed in time) started in $v$, the claim readily follows from the fact that the Poisson-processes are time-reversible. To be more precise, the SAD-process started in $u$ until time $t$ and the dual SAD-process (with respect to the time interval $[0,t]$) started in $u$ have the same distribution. While the former leads to the water level $\xi_t(u,v)$ at vertex $v$, the value at vertex $v$ at time $t$ in the latter is $\xi_t(v,u)$.
\end{proof}

Next, we prove the fact that the contributions $\xi_t(v,u)$ maximized over either sender or recipient almost surely converge to 0.

\begin{lemma}\label{sender}
	Let $X_t(u)=\max\{\xi_t(u,v),\; v\in V\}$ denote the maximal contribution by vertex $u$ to another vertex at time $t$. Then $X_t(u)\stackrel{\text{a.s.}}{\longrightarrow} 0$, as $t\to\infty$.
\end{lemma}

\begin{proof}
	As above, we consider $\{\xi_t(u,v),\,v\in V\}$ as the configuration at time $t$ of the averaging process on $G$ in continuous time started with $\eta_0=\delta_u$, which was earlier denoted by $\{\eta_t(u),\,u\in V\}$. 
	Adopting this notation we can follow the lines of argument in the proof of Lemma \ref{finite}:
	Let $W_t=\sum_{v\in V}\eta_t(v)^2$ denote the total energy at time $t$. Then both $X_t(u)$ and $W_t$ are piecewise constant, right-continuous (random) functions and decreasing in $t$. Moreover, they are $[0,1]$-valued as $\eta_0=\delta_u$ implies $X_0(u)=W_0=1$. As a consequence, the pointwise limit \[X(u):=\lim_{t\to\infty}X_t(u)\] exists (in fact not only almost surely, but for \emph{all} $\omega$).
	
	Let us assume for contradiction that there exists $\epsilon>0$ such that the event $\{X(u)\geq\epsilon\}$ has positive
	probability. By monotonicity it holds $X_t(u)\geq X(u)$ for all $t\geq 0$. Due to the simple fact that the sum $\sum_{v\in V}\eta_t(v)=1$ is constant in
	time, the set $\{v\in V,\,\eta_t(v)\geq \frac\epsilon2\}$ can not contain more than $\lfloor\frac2\epsilon\rfloor$ vertices. Consequently, given $\{X(u)\geq\epsilon\}$, there has to be a gap in $[\frac\epsilon2,\epsilon]$ of width at least $\frac{\epsilon^2}{4}$ containing no values of $\{\eta_t(v),\; v\in V\}$. Since there are nodes with values above the gap (due to $X_t(u)\geq \epsilon$) and nodes with values below (in fact $\eta_t(v)=0$ for all but finitely many $v\in V$), there will a.s.\ be an update along an edge with incident vertices having values on different sides of the gap after time $t$. Such an update reduces the total energy by an amount of at least  $\mu(1-\mu)\frac{\epsilon^4}{8}$. As this can happen only finitely many times, we arrive again at a contradiction proving $X(u)\equiv0$.
\end{proof}

Despite the fact that for arbitrary $t$ it holds $\max\{\xi_t(u,v),\; u,v\in V\}=1$ with probability 1 on an infinite, connected graph $G$, the maximal contribution to a fixed vertex $v$ converges to 0 almost surely:

\begin{lemma}\label{recipient}
	Let $Y_t(v)=\max\{\xi_t(u,v),\; u\in V\}$ denote the maximal contribution to vertex $v$ by another vertex at time $t$. Then $Y_t(v)\stackrel{\text{a.s.}}{\longrightarrow} 0$, as $t\to\infty$.
\end{lemma}

\begin{proof}
	Fix arbitrary $\epsilon>0$ and consider the ball \[B_{\frac1\epsilon}(v)=\{u\in V,\, d(u,v)\leq\tfrac1\epsilon\},\]
	which contains finitely many vertices (recall that we assume bounded degree for $G$).
	Then we know from Theorem \ref{globalbound} that \[Y_t(v)\leq\max\big\{\epsilon,\max\{\xi_t(u,v),\; u\in B_{\frac1\epsilon}(v)\} \big\}.\]
	From Lemma \ref{sender} we know that $\xi_t(u,v)\leq X_t(u)$ converges to $0$ almost surely for every fixed $u$. As the maximum is taken over $B_{\frac1\epsilon}(v)$, i.e.\ a finite number of such random variables, we can conclude
	$\limsup_{t\to\infty}Y_t(v)\leq \epsilon$ a.s., which concludes the proof as the $Y_t$ are positive and $\epsilon>0$ was arbitrary.
\end{proof}

In fact, we proved a slightly stronger result:
\begin{corollary}\label{correcip}
	For arbitrary but fixed $v\in V$, it also holds that \[\max_{u\in V}\sup_{s\geq t}\xi_s(u,v)\stackrel{\text{a.s.}}{\longrightarrow} 0,\text{ as }t\to\infty.\]
\end{corollary}
\begin{proof}
	This claim follows immediately from the line of reasoning above, using the fact that not only $\xi_t(u,v)$ but also $\sup_{s\geq t}\xi_s(u,v)$ is dominated by $X_t(u)=\max\{\xi_t(u,v),\; v\in V\}$, due to the fact that $X_t(u)$ is non-increasing in $t$.	
\end{proof}

\noindent
We are now ready to prove our main result.\\

\noindent
\begin{nproof}{of Theorem \ref{avgthm}}
	To begin with, we can assume w.l.o.g.\ that $m_1=0$ (by subtracting the mean $m_1$ from all values if need be). As before, let $\mathcal{F}_t$ denote the sigma field generated by all Poisson-processes until time $t\geq0$. Using \eqref{convcomb} and recalling \eqref{defsecmom} we find
	\begin{align*} \IE[\eta_t^2(v)\,|\,\mathcal{F}_t]&=\sum_{u\in V}\xi_t^2(u,v)\,m_2\\&\leq \max_{u\in V}\xi_t(u,v)\cdot m_2\cdot\underbrace{\sum_{u\in V}\xi_t(u,v)}_{=1},
	\end{align*}
	where dominated convergence, $m_1=0$ and the independence of $\{\eta_0(v)\}_{v\in V}$ and $\mathcal{F}_t$ were used in the first step. Using the tower property we arrive at
	\[ \IE\big(\eta_t^2(v)\big)=\IE\big(\IE[\eta_t^2(v)\,|\,\mathcal{F}_t]\big) \leq m_2\cdot\IE\Big(\max_{u\in V}\xi_t(u,v)\Big)\]
	Finally, from Lemma \ref{recipient} together with uniform boundedness of the $\xi_t(u,v)$, i.e.\ $0\leq\xi_t(u,v)\leq1$ for all $u,v\in V,\ t\geq0$, it follows that $Y_t(v)=\max_{u\in V}\xi_t(u,v)$ converges to $0$ also in $L^1$, which proves the claim.
\end{nproof}

\noindent
{\sc Acknowledgements} We thank Stefan Weltge for transmitting our question to the authors of \cite{SAD} and Tony Huynh for updating us on their progress.


\vspace{0.5cm}\noindent
	{\sc \small 
	Nina Gantert\\
	SoCIT, Department of Mathematics,\\
	Technical University of Munich,\\
	Boltzmannstr. 3, 85748 Garching\\
	Germany}\\
	gantert@ma.tum.de\\

\vspace{0.5cm}\noindent
	{\sc \small 
	Timo Vilkas\\
	Statistiska institutionen,\\
	Ekonomih\"ogskolan vid Lunds universitet,\\
	220\,07 Lund\\
	Sweden}\\
	timo.vilkas@stat.lu.se\\

\end{document}